\documentclass[12pt]{article}
\usepackage[latin1]{inputenc}
\usepackage[british]{babel}
\usepackage{lmodern}
\usepackage[T1]{fontenc}
\usepackage[paper=a4paper, left=28mm, right=25mm, top=29mm, bottom=25mm]{geometry}
\usepackage{latexsym,amsfonts,amsmath,graphics}
\usepackage{epsfig}
\usepackage{enumitem}
\usepackage{amssymb,mathrsfs}
\usepackage{verbatim}
\newtheorem{theorem}{Theorem}
\newtheorem{lemma}{Lemma}

\newtheorem{proposition}{Proposition}

\newtheorem{remark}{Remark}

\newenvironment{proof}{\begin{trivlist} \item[\hskip\labelsep{\it Proof.}]}{$\hfill\Box$\end{trivlist}}

\newcommand{\gr}{\gtrsim}
\newcommand{\lr}{\lesssim}

\newcommand{\rd}{\,\mathrm{d}}
\newcommand{\bsj}{\boldsymbol{j}}
\newcommand{\bsx}{\boldsymbol{x}}

\newcommand{\bsz}{\boldsymbol{z}}

\newcommand{\bsm}{\boldsymbol{m}}
\newcommand{\bst}{\boldsymbol{t}}
\newcommand{\bsa}{\boldsymbol{a}}

\newcommand{\bszero}{\boldsymbol{0}}
\newcommand{\bsone}{\boldsymbol{1}}

\newcommand{\RR}{\mathbb{R}}

\newcommand{\NN}{\mathbb{N}}

\newcommand{\DD}{\mathbb{D}}

\newcommand{\ZZ}{\mathbb{Z}}

\newcommand{\cP}{\mathcal{P}}

\newcommand{\cH}{\mathcal{H}}

\allowdisplaybreaks

\title{Digital nets in dimension two with the optimal order of $L_p$ discrepancy}
\author{Ralph Kritzinger and Friedrich Pillichshammer\thanks{The authors are supported by the Austrian Science Fund (FWF): Project F5509-N26, which is a part of the Special Research Program "Quasi-Monte Carlo Methods: Theory and Applications".}}
\date{}

\begin{document}

\maketitle

\begin{abstract}
We study the $L_p$ discrepancy of two-dimensional digital nets for finite $p$. In the year 2001 Larcher and Pillichshammer identified a class of digital nets for which the symmetrized version in the sense of Davenport has $L_2$ discrepancy of the order $\sqrt{\log N}/N$, which is best possible due to the celebrated result of Roth. However, it remained open whether this discrepancy bound also holds for the original digital nets without any modification. 

In the present paper we identify nets from the above mentioned class for which the symmetrization is not necessary in order to achieve the optimal order of $L_p$ discrepancy for all $p \in [1,\infty)$. 

Our findings are in the spirit of a paper by Bilyk from 2013, who considered the $L_2$ discrepancy of lattices consisting of the elements $(k/N,\{k \alpha\})$ for $k=0,1,\ldots,N-1$, and who gave Diophantine properties of $\alpha$ which guarantee the optimal order of $L_2$ discrepancy. 
\end{abstract}

\centerline{\begin{minipage}[hc]{130mm}{
}
\end{minipage}}

 \allowdisplaybreaks

\section{Introduction}
Discrepancy is a measure for the irregularities of point distributions in the unit interval (see, e.g., \cite{kuinie}).
Here we study point sets $\cP$ with $N$ elements in the two-dimensional unit interval $[0,1)^2$. We define
the {\it discrepancy function} of such a point set by
$$ \Delta_{\cP}(\bst)=\frac{1}{N}\sum_{\bsz\in\cP}\bsone_{[\bszero,\bst)}(\bsz)-t_1t_2, $$
where for $\bst=(t_1,t_2)\in [0,1]^2$ we set $[\bszero,\bst)=[0,t_1)\times [0,t_2)$ with area $t_1t_2$
and denote by $\bsone_{[\bszero,\bst)}$ the indicator function of this interval. The {\it $L_p$ discrepancy} for $p\in [1,\infty)$ of $\cP$ is given by
$$ L_{p}(\cP):=\|\Delta_{\cP}\|_{L_{p}([0,1]^2)}=\left(\int_{[0,1]^2}|\Delta_{\cP}(\bst)|^p\rd \bst\right)^{\frac{1}{p}} $$
and the {\it star discrepancy} or {\it $L_{\infty}$ discrepancy} of $\cP$ is defined as
$$ L_{\infty}(\cP):=\|\Delta_{\cP}\|_{L_{\infty}([0,1]^2)}=\sup_{\bst \in [0,1]^2}|\Delta_{\cP}(\bst)|. $$

The $L_p$ discrepancy is a quantitative measure for the irregularity of distribution of a point set. Furthermore, it is intimately related to the worst-case integration error of quasi-Monte Carlo rules; see \cite{DP10,kuinie, LP14,Nied92}.

It is well known that for every $p\in [1,\infty)$ we have\footnote{Throughout this paper, for functions $f,g:\NN \rightarrow \RR^+$, we write $g(N) \lr f(N)$, 
if there exists a $C>0$ such that $g(N) \le C f(N)$ with a positive constant $C$ that is independent of $N$. Likewise, we write $g(N) \gr f(N)$ if $g(N) \geq C f(N)$. Further, we write $f(N) \asymp g(N)$ if the relations $g(N) \lr f(N)$ and $g(N) \gr f(N)$ hold simultaneously.} 
\begin{equation} \label{roth} 
L_p(\mathcal{P}) \gr_p \frac{\sqrt{\log{N}}}{N}, 
\end{equation}
for every $N \ge 2$ and every $N$-element point set $\mathcal{P}$ in $[0,1)^2$. Here $\log$ denotes the natural logarithm. This was first shown by Roth \cite{Roth2} for $p = 2$ and hence for all $p \in [2,\infty]$ and later by
Schmidt \cite{schX} for all $p\in(1,2)$. The case $p=1$ was added by Hal\'{a}sz \cite{hala}. For the star discrepancy we have according to Schmidt~\cite{Schm72distrib} that
\begin{equation} \label{schmidt} 
L_{\infty}(\mathcal{P}) \gr \frac{\log{N}}{N}, 
\end{equation}
for every $N \ge 2$ and every $N$-element point set $\mathcal{P}$ in $[0,1)^2$.

\paragraph{Irrational lattices.}
It is well-known, that the lower bounds in \eqref{roth} and \eqref{schmidt} are best possible in the order of magnitude in $N$. For example, when the irrational number $\alpha=[a_0;a_1,a_2,\ldots]$ has bounded partial quotients in it's continued fraction expansion, then the lattice $\cP_{\alpha}$ consisting of the points $(k/N,\{k \alpha\})$ for $k=0,1,\ldots,N-1$, where $\{\cdot\}$ denotes reduction modulo one, has optimal order of star discrepancy in the sense of \eqref{schmidt} (see, e.g., \cite{lerch} or \cite[Corollary~3.5 in combination with Lemma~3.7]{Nied92}). This is, in this generality, not true anymore when, e.g., the $L_2$ discrepancy is considered. However, in 1956 Davenport~\cite{daven} showed that the symmetrized version $\cP_{\alpha}^{{\rm sym}}:=\cP_{\alpha}\cup \cP_{-\alpha}$ of $\cP_{\alpha}$ consisting of $2N$ points has $L_2$ discrepancy of the order $\sqrt{\log N}/N$ which is optimal with respect to \eqref{roth}. 

Later Bilyk~\cite{bil} introduced a further condition on $\alpha$ which guarantees the optimal order of $L_2$ discrepancy without the process of symmetrization. If and only if the bounded partial quotients satisfy $|\sum_{k=0}^{N-1} (-1)^k a_k| \lr_{\alpha} \sqrt{n}$, then $L_2(\cP_{\alpha}) \asymp_{\alpha} \sqrt{\log N}/N$.

\paragraph{Digital nets.} In this paper we study analog questions for digital nets over $\ZZ_2$, which are an important class of point sets with low star discrepancy. Since we only deal with digital nets over $\ZZ_2$ and in dimension 2 we restrict the necessary definitions to this case. For the general setting we refer to the books of Niederreiter~\cite{Nied92} (see also \cite{Nied87}), of Dick and Pillichshammer~\cite{DP10}, or of Leobacher and Pillichshammer~\cite{LP14}.

Let $n\in \NN$ and let $\ZZ_2$ be the finite field of order 2, which we identify with the set $\{0,1\}$ equipped with arithmetic operations modulo 2. A two-dimensional digital net over $\ZZ_2$ is a point set $\{\boldsymbol{x}_0,\ldots, \boldsymbol{x}_{2^n-1}\}$ in $[0,1)^2$, which is generated by two $n\times n$ matrices over $\ZZ_2$. The procedure is as follows. 
\begin{enumerate}
\item Choose two $n \times n$ matrices $C_1$ and $C_2$ with entries from $\ZZ_2$.
\item For $r\in\{0,1,\dots,2^n-1\}$ let $r=r_0+2r_1  +\cdots +2^{n-1}r_{n-1}$ with $r_i\in\{0,1\}$ for all $i\in\{0,\dots,n-1\}$ be the dyadic expansion of $r$, and set $\vec{r}=(r_0,\ldots,r_{n-1})^{\top}\in \ZZ_2^n$. 
\item For $j=1,2$ compute $C_j \vec{r}=:(y_{r,1}^{(j)},\ldots ,y_{r,n}^{(j)})^{\top}\in \ZZ_2^n$, where all arithmetic operations are over $\ZZ_2$.
\item For $j=1,2$ compute $x_r^{(j)}=\frac{y_{r,1}^{(j)}}{2}+\cdots +\frac{y_{r,n}^{(j)}}{2^n}$ and set $\boldsymbol{x}_{r}=(x_r^{(1)},x_{r}^{(2)})\in [0,1)^2$.
\item Set $\cP:=\{\bsx_0,\dots,\bsx_{2^n-1}\}$. We call $\cP$ a {\it digital net over $\ZZ_2$} generated by $C_1$ and $C_2$.
\end{enumerate}

One of the most well-known digital nets is the {\it 2-dimensional Hammersley net $\cP^{{\rm Ham}}$ in base 2} which is generated by the matrices 
$$C_1 = \left ( \begin{array}{llcll}
0 & 0 & \cdots & 0 & 1\\
0 & 0 & \cdots & 1 & 0 \\
\multicolumn{5}{c}\dotfill\\
0 & 1 & \cdots & 0 & 0 \\
1 & 0 & \cdots & 0 & 0
\end{array} \right ) \ \ \mbox{ and } \ \
C_2 =  \left ( \begin{array}{llcll}
1 & 0 & \cdots & 0 & 0\\
0 & 1 & \cdots & 0 & 0 \\
\multicolumn{5}{c}\dotfill\\
0 & 0 & \cdots & 1 & 0 \\
0 & 0 & \cdots & 0 & 1
\end{array} \right ).$$
Due to the choice of $C_1$ the first coordinates of the elements of the Hammersley net are $x_r^{(1)}=r/2^n$ for $r=0,1,\ldots,2^n -1$.

\paragraph{$(0,n,2)$-nets in base 2.}
A point set $\cP$ consisting of $2^n$ elements in $[0,1)^2$ is called a {\it $(0,n,2)$-net in base 2}, if every dyadic box 
$$\left[\frac{m_1}{2^{j_1}},\frac{m_1+1}{2^{j_1}}\right) \times \left[\frac{m_2}{2^{j_2}},\frac{m_2+1}{2^{j_2}}\right),$$
where $j_1,j_2\in\NN_0$ and $m_1\in\{0,1,\dots,2^{j_1}-1\}$ and $m_2\in\{0,1,\dots,2^{j_2}-1\}$  with volume $2^{-n}$, i.e. with $j_1+j_2=n$, contains exactly one element of $\cP$. 

It is well known that a digital net over $\ZZ_2$ is a $(0,n,2)$-net in base 2 if and only if the following condition holds: For every choice of integers $d_1,d_2\in \NN_0$ with $d_1+d_2=n$ the first $d_1$ rows of $C_1$ and the first $d_2$ rows of $C_2$ are linearly independent.

Every digital $(0,n,2)$-net achieves the optimal order of star discrepancy in the sense of \eqref{schmidt}, whereas there exist nets which do not have the optimal order of $L_p$ discrepancy for finite $p$. One example is the Hammersley net as defined above for which we have (see \cite{FauPil,Lar,Pill}) $$L_p(\cP^{{\rm Ham}})=\left(\left(\frac{n}{8 \cdot 2^n}\right)^p+O(n^{p-1})\right)^{1/p} \ \ \mbox{for all $p\in [1,\infty)$}$$ and $$L_{\infty}(\cP^{{\rm Ham}})=\frac{1}{2^n} \left(\frac{n}{3}+\frac{13}{9}-(-1)^n \frac{4}{9 \cdot 2^n}\right).$$

\paragraph{Symmetrized nets.}

Motivated by the results of Davenport for irrational lattices, Larcher and Pillichshammer~\cite{lp01}  studied the symmetrization of digital nets. Let $\bsx_r=(x_r,y_r)$ for $r=0,1,\ldots,2^n-1$ be the elements of a digital net generated by the matrices $$C_1 = \left ( \begin{array}{llcll}
0 & 0 & \cdots & 0 & 1\\
0 & 0 & \cdots & 1 & 0 \\
\multicolumn{5}{c}\dotfill\\
0 & 1 & \cdots & 0 & 0 \\
1 & 0 & \cdots & 0 & 0
\end{array} \right )\ \ \mbox{ and }\ \ C_2 =  \left ( \begin{array}{llcll}
1 & a_{1,2} & \cdots & a_{1,n-1} & a_{1,n}\\
0 & 1 & \cdots & a_{2,n-1} & a_{2,n} \\
\multicolumn{5}{c}\dotfill\\
0 & 0 & \cdots & 1 & a_{n-1,n} \\
0 & 0 & \cdots & 0 & 1
\end{array} \right ),
$$
with entries $a_{j,k} \in \ZZ_2$ for $1 \le j <k \le n$. The matrix $C_2$ is a so-called ``{\it non-singular upper triangular (NUT) matrix}''. Then the {\it symmetrized net} $\cP^{{\rm sym}}$ consisting of $(x_r,y_r)$ and $(x_r,1-y_r)$ for $r=0,1,\ldots,2^n-1$ has $L_2$ discrepancy of optimal order $$L_2(\cP^{{\rm sym}}) \asymp \frac{\sqrt{n}}{2^{n+1}} \ \ \ \mbox{for every $n \in \NN$.}$$

In the present paper we show in the spirit of the paper of Bilyk~\cite{bil} that there are NUT matrices $C_2$ such that symmetrization is not required in order to achieve the optimal order of $L_2$ discrepancy. Or result we be true for the $L_p$ discrepancy for all finite $p$ and not only for the $L_2$ case.

\section{The result}

The central aim of this paper is to provide conditions on the generating matrices $C_1,C_2$ which lead to the optimal order of $L_p$ discrepancy of the corresponding nets. We do so for a class of nets which are generated by $n\times n$ matrices over $\ZZ_2$ of the following form:
\begin{equation} \label{matrixa} C_1=
\begin{pmatrix}
0 & 0 & \cdots & 0 & 1\\
0 & 0 & \cdots & 1 & 0 \\
\multicolumn{5}{c}\dotfill\\
0 & 1 & \cdots & 0 & 0 \\
1 & 0 & \cdots & 0 & 0
\end{pmatrix}
\end{equation}
and a NUT matrix of the special form
\begin{equation} C_2=
\begin{pmatrix}
1 & a_{1} & a_{1} & \cdots & a_{1} & a_{1} & a_{1} \\
0 & 1 & a_{2} & \cdots &  a_{2} & a_{2} & a_{2} \\
0 & 0 & 1 & \cdots & a_{3} & a_{3} & a_{3} \\
\vdots & \vdots & \vdots & \ddots & \vdots & \vdots & \vdots & \\
0 & 0 & 0 & \cdots &  1 & a_{n-2} & a_{n-2} \\
0 & 0 & 0 & \cdots &  0 & 1 & a_{n-1} \\
0 & 0 & 0 & \cdots &  0 & 0 & 1 
\end{pmatrix},
\end{equation}
where $a_i\in \ZZ_2$ for all $i\in\{1,\dots,n-1\}$. We study the $L_p$ discrepancy of the digital net $\cP_{\bsa}$ generated by $C_1$ and $C_2$, where $\bsa=(a_1,\dots,a_{n-1})\in \ZZ_2^{n-1}$. The set $\cP_{\bsa}$ can be written as
\begin{equation}\label{darstPa}
\cP_{\bsa}=\left\{\bigg(\frac{t_n}{2}+\dots+\frac{t_1}{2^n},\frac{b_1}{2}+\dots+\frac{b_n}{2^n}\bigg):t_1,\dots, t_n \in\{0,1\}\right\}, 
\end{equation}
where $b_k=t_k\oplus a_{k}(t_{k+1}\oplus \dots \oplus t_n)$ for $k\in\{1,\dots,n-1\}$ and $b_n=t_n$. The operation $\oplus$ denotes addition modulo 2.

The following result states that the order of the $L_p$ discrepancy of the digital nets $\cP_{\bsa}$ is determined by the number of zero elements in $\bsa$.

\begin{theorem} \label{theo1}
Let $h_n=h_n(\bsa)=\sum_{i=1}^{n-1}(1-a_i)$ be the number of zeroes in the tuple $\bsa$. Then we have for all $p\in[1,\infty)$
$$L_p(\cP_{\bsa})\asymp_p \frac{\max\{\sqrt{n},h_n(\bsa)\}}{2^n}.$$
In particular, the net $\cP_{\bsa}$ achieves the optimal order of $L_p$ discrepancy for all $p\in [1,\infty)$ if and only if $h_n(\bsa)\lesssim \sqrt{n}$.
\end{theorem}
 
The proof of Theorem~\ref{theo1}, which will be given in Section~\ref{haarf}, is based on Littlewood-Paley theory and tight estimates of the Haar coefficients of the discrepancy function $\Delta_{\cP_{\bsa}}$. 

For example, if $\bsa=\bszero:=(0,0,\ldots,0)$ we get the Hammersley net $\cP^{{\rm Ham}}$ in dimension 2. We have $h_n(\bszero)=n-1$ and hence $$L_p(\cP_{\bszero})\asymp_p \frac{n}{2^n}.$$ If $\bsa=\bsone:=(1,1,\ldots,1)$, then we have $h_n(\bsone)=0$ and hence $$L_p(\cP_{\bsone})\asymp_p \frac{\sqrt{n}}{2^n}.$$

\begin{remark} \rm
The approach via Haar functions allows the precise computation of the $L_2$ discrepancy of digital nets via Parseval's identity. We did so for a certain class of nets in~\cite{Kritz}. It would be possible but tedious to do the same for the class $\cP_{\bsa}$ of nets considered in this paper. However, we only executed the massive calculations for the special case where $\bsa=\bsone:=(1,1,\dots,1)$, hence where $C_2$ is a NUT matrix filled with ones in the upper right triangle. We conjecture that this net has the lowest $L_2$ discrepancy among the class of nets $\cP_{\bsa}$ for a fixed $n\in\NN$. The exact value of its $L_2$ discrepancy is given by 
\begin{equation}\label{LpP1}
L_2(\cP_{\bsone})=\frac{1}{2^n}\left(\frac{5n}{192}+\frac{15}{32}+\frac{1}{4\cdot 2^{n}}-\frac{1}{72\cdot 2^{2n}}\right)^{1/2}. 
\end{equation}
We omit the lengthy proof, but its correctness may be checked with Warnock's formula~\cite{Warn} (see also \cite[Proposition~2.15]{DP10})for small values of $n$. Compare \eqref{LpP1} with the exact $L_2$ discrepancy of $\cP^{{\rm Ham}}=\cP_{\bszero}$ which is given by (see \cite{FauPil,HaZa,Pill,Vi}) 
$$L_2(\cP_{\bszero})=\frac{1}{2^n}\left(\frac{n^2}{64}+\frac{29n}{192}+\frac{3}{8}-\frac{n}{16 \cdot 2^n}+\frac{1}{4\cdot 2^n}-\frac{1}{72 \cdot 2^{2n}}\right)^{1/2}.$$
\end{remark}

\section{The proof of Theorem~\ref{theo1} via Haar expansion of the discrepancy function} \label{haarf}

A dyadic interval of length $2^{-j}, j\in {\mathbb N}_0,$ in $[0,1)$ is an interval of the form 
$$ I=I_{j,m}:=\left[\frac{m}{2^j},\frac{m+1}{2^j}\right) \ \ \mbox{for } \  m\in \{0,1,\ldots,2^j-1\}.$$ 
The left and right half of $I_{j,m}$ are the dyadic intervals $I_{j+1,2m}$ and $I_{j+1,2m+1}$, respectively. The Haar function $h_{j,m}$  
is the function on $[0,1)$ which is  $+1$ on the left half of $I_{j,m}$, $-1$ on the right half of $I_{j,m}$ and 0 outside of $I_{j,m}$. The $L_\infty$-normalized Haar system consists of
all Haar functions $h_{j,m}$ with $j\in{\mathbb N}_0$ and  $m=0,1,\ldots,2^j-1$ together with the indicator function $h_{-1,0}$ of $[0,1)$.
Normalized in $L_2([0,1))$ we obtain the orthonormal Haar basis of $L_2([0,1))$. 

Let ${\mathbb N}_{-1}=\NN_0 \cup \{-1\}$ and define ${\mathbb D}_j=\{0,1,\ldots,2^j-1\}$ for $j\in{\mathbb N}_0$ and ${\mathbb D}_{-1}=\{0\}$.
For $\bsj=(j_1,j_2)\in{\mathbb N}_{-1}^2$ and $\bsm=(m_1,m_2)\in {\mathbb D}_{\bsj} :={\mathbb D}_{j_1} \times {\mathbb D}_{j_2}$, 
the Haar function $h_{\bsj,\bsm}$ is given as the tensor product 
$$h_{\bsj,\bsm}(\bst) = h_{j_1,m_1}(t_1) h_{j_2,m_2}(t_2) \ \ \ \mbox{ for } \bst=(t_1,t_2)\in[0,1)^2.$$
We speak of $I_{\bsj,\bsm} = I_{j_1,m_1} \times I_{j_2,m_2}$ as dyadic boxes with level $|\bsj|=\max\{0,j_1\}+\max\{0,j_2\}$, where we set $I_{-1,0}=\bsone_{[0,1)}$. The system
$$ \left\{2^{\frac{|\bsj|}{2}}h_{\bsj,\bsm}: \bsj\in\NN_{-1}^2, \bsm\in \DD_{\bsj}\right\} $$
is an orthonormal basis of $L_2([0,1)^2)$ and we have Parseval's identity which states that for every function $f\in L_2([0,1)^2)$ we have
\begin{equation} \label{parseval}
   \|f\|_{L_2([0,1)^2)}^2=\sum_{\bsj\in \NN_{-1}^2} 2^{|\bsj|} \sum_{\bsm\in\DD_{\bsj}} |\mu_{\bsj,\bsm}|^2,
\end{equation}
where the numbers $\mu_{\bsj,\bsm}=\mu_{\bsj,\bsm}(f)=\langle f, h_{\bsj,\bsm} \rangle =\int_{[0,1)^2} f(\bst) h_{\bsj,\bsm}(\bst)\rd\bst$ are the so-called Haar coefficients of $f$. There
is no such identity for the $L_p$ norm of $f$ for $p \not=2$; however, for a function $f\in L_p([0,1)^2)$ we have a so-called Littlewood-Paley inequality. It involves the square function $S(f)$ of a function $f\in L_p([0,1)^2)$ which is given as
$$S(f) = \left( \sum_{\bsj \in \NN_{-1}^2} \sum_{\bsm \in \mathbb{D}_{\bsj}} 2^{2|\bsj|} \, |\mu_{\bsj,\bsm}|^2 \, {\mathbf 1}_{I_{\bsj,\bsm}} \right)^{1/2},$$ where ${\mathbf 1}_I$ is the characteristic function of $I$.

\begin{lemma}[Littlewood-Paley inequality]\label{lpi}
 Let $p \in (1,\infty)$ and let $f\in L_p([0,1)^2)$. Then 
 $$ \| S(f) \|_{L_p} \asymp_{p} \| f \|_{L_p}.$$
\end{lemma}

In the following let $\mu_{\bsj,\bsm}$ denote the Haar coefficients if the local discrepancy function $\Delta_{\cP_{\bsa}}$, i.e., $$\mu_{\bsj,\bsm}=\int_{[0,1)^2} \Delta_{\cP_{\bsa}}(\bst) h_{\bsj,\bsm}(\bst) \rd \bst.$$  In order to estimate the $L_p$ discrepancy of $\cP_{\bsa}$ by means of Lemma~\ref{lpi} we require good estimates of the Haar coefficients $\mu_{\bsj,\bsm}$. This is a very technical and tedious task which we defer to the appendix. In the following we just collect the obtained bounds:

\begin{lemma} \label{coro1}
Let $\bsj=(j_1,j_2)\in \NN_{0}^2$. Then
 \begin{itemize}
  \item[(i)] if $j_1+j_2\leq n-3$ and $j_1,j_2\geq 0$ then $|\mu_{\bsj,\bsm}| \lr 2^{-2n}$.
  \item[(ii)] if $j_1+j_2\ge n-2$ and $0\le j_1,j_2\le n$ then $|\mu_{\bsj,\bsm}| \lr 2^{-n-j_1-j_2}$ and
     $|\mu_{\bsj,\bsm}| = 2^{-2j_1-2j_2-4}$ for all but at most $2^n$ coefficients $\mu_{\bsj,\bsm}$ with $\bsm\in {\mathbb D}_{\bsj}$.
  \item[(iii)] if $j_1 \ge n$ or $j_2 \ge n$ then $|\mu_{\bsj,\bsm}| = 2^{-2j_1-2j_2-4}$.
 \end{itemize} 
 Now let $\bsj=(-1,k)$ or $\bsj=(k,-1)$ with $k\in \NN_0$. Then
 \begin{itemize}
  \item[(iv)] if $k<n$ then $|\mu_{\bsj,\bsm}| \lr 2^{-n-k}$.
  \item[(v)] if $k\ge n$  then $|\mu_{\bsj,\bsm}| = 2^{-2k-3}$.
 \end{itemize}
 Finally, if $h_n=\sum_{i=1}^{n-1}(1-a_i)$, then
  \begin{itemize}
   \item[(vi)] $\mu_{(-1,-1),(0,0)} = 2^{-n-3}(h_n+5)+2^{-2n-2}$.
  \end{itemize} 
\end{lemma}

\begin{remark}\rm
We remark that Proposition~\ref{coro1} shows that the only Haar coefficient that is relevant in our analysis is the coefficient $\mu_{(-1,-1),(0,0)}$. All other coefficients do not affect the order of $L_p$ discrepancy significantly: they are small enough such that their contribution to the over all $L_p$ discrepancy is of the order of Roth's lower bound.  

The proof of Proposition~\ref{coro1} is split into several cases which take several pages of very technical and tedious computations. We would like to mention that the proof of the formula for the important coefficient $\mu_{(-1,-1),(0,0)}$ is manageable without excessive effort. 
\end{remark}

Now the proof of Theorem~\ref{theo1} can be finished by inserting the upper bounds on the Haar coefficients of $\Delta_{\cP_{\bsa}}$ into Lemma~\ref{lpi}. This shows the upper bound. For details we refer to the paper \cite{HKP14} where the same method was applied (we remark that our Proposition~\ref{coro1} is a direct analog of \cite[Lemma~1]{HKP14}; hence the proof of Theorem~\ref{theo1} runs along the same lines as the proof of
\cite[Theorem 1]{HKP14} but with \cite[Lemma~1]{HKP14} replaced by Proposition~\ref{coro1}).

The matching lower bound is a consequence of $$L_p(\cP_{\bsa}) \ge L_1(\cP_{\bsa}) =\int_{[0,1]^2} | \Delta_{\cP_{\bsa}}(\bst)| \rd \bst \ge \left|\int_{[0,1]^2} \Delta_{\cP_{\bsa}}(\bst) \rd \bst\right|=|\mu_{(-1,-1),(0,0)}|$$ and item {\it (vi)} of Lemma~\ref{coro1}. 

\section{Appendix: Computation of the Haar coefficients $\mu_{\bsj,\bsm}$}

Let $\cP$ be an arbitrary $2^n$-element point set in the unit square. The Haar coefficients of its discrepancy function $\Delta_{\cP}$ are given as follows (see~\cite{hin2010}). We write $\bsz=(z_1,z_2)$. 

\begin{itemize}
  \item If $\bsj=(-1,-1)$, then 
   \begin{equation} \label{art1} \mu_{\bsj,\bsm}=\frac{1}{2^n}\sum_{\bsz\in \cP} (1-z_1)(1-z_2)-\frac14. \end{equation}
   \item If $\bsj=(j_1,-1)$ with $j_1\in \NN_0$, then 
   \begin{equation} \label{art2} \mu_{\bsj,\bsm}=-2^{-n-j_1-1}\sum_{\bsz\in \cP\cap I_{\bsj,\bsm}} (1-|2m_1+1-2^{j_1+1}z_1|)(1-z_2)+2^{-2j_1-3}. \end{equation}
    \item If $\bsj=(-1,j_2)$ with $j_2\in \NN_0$, then 
   \begin{equation} \label{art3} \mu_{\bsj,\bsm}=-2^{-n-j_2-1}\sum_{\bsz\in \cP\cap I_{\bsj,\bsm}} (1-|2m_2+1-2^{j_2+1}z_2|)(1-z_1)+2^{-2j_2-3}. \end{equation}
    \item If $\bsj=(j_1,j_2)$ with $j_1,j_2\in \NN_0$, then 
   \begin{align} \label{art4} \mu_{\bsj,\bsm}=&2^{-n-j_2-j_2-2}\sum_{\bsz\in \cP\cap I_{\bsj,\bsm}} (1-|2m_1+1-2^{j_1+1}z_1|)(1-|2m_2+1-2^{j_2+1}z_2|) \nonumber \\ &-2^{-2j_1-2j_2-4}. \end{align}
   \end{itemize}
	In all these identities the first summands involving the sum over $\bsz\in \cP\cap I_{\bsj,\bsm}$ come from the counting part $\frac{1}{N}\sum_{\bsz\in\cP}\bsone_{[\bszero,\bst)}(\bsz)$  and the second summands come from the linear part $-t_1t_2$ of the discrepancy function, respectively.
Note that we could also write $\bsz\in \mathring{I}_{\bsj,\bsm}$, where $\mathring{I}_{\bsj,\bsm}$ denotes the interior of $I_{\bsj,\bsm}$, since the summands in the formulas~\eqref{art2}--\eqref{art4} vanish if $\bsz$ lies on the boundary of the dyadic box. Hence, in order to compute the Haar coefficients of the discrepancy function, we have to deal with the sums over $\bsz$ which appear in the formulas above and to determine which points $\bsz=(z_1,z_2)\in \cP$ lie in the dyadic box $I_{\bsj,\bsm}$ with $\bsj\in \NN_{-1}^2$ and $\bsm=(m_1,m_2)\in\DD_{\bsj}$. If $m_1$ and $m_2$ are non-negative integers, then they have a dyadic expansion of the form
\begin{equation} \label{mdyadic} m_1=2^{j_1-1}r_1+\dots+r_{j_1}  \text{\,  and  \,}  m_2=2^{j_2-1}s_1+\dots+s_{j_2}  \end{equation}
with digits $r_{i_1},s_{i_2}\in\{0,1\}$ for all $i_1\in\{1,\dots,j_1\}$ and $i_2\in\{1,\dots,j_2\}$, respectively.
Let $\bsz=(z_1,z_2)=\big(\frac{t_n}{2}+\dots+\frac{t_1}{2^n},\frac{b_1}{2}+\dots+\frac{b_n}{2^n}\big)$ be a point of our point set $\cP_{\bsa}$. Then $\bsz\in \cP_{\bsa}\cap I_{\bsj,\bsm}$
if and only if 
\begin{equation} \label{cond} t_{n+1-k}=r_k \text{\, for all \,} k\in \{1,\dots, j_1\} \text{\, and \,} b_k=s_k \text{\, for all \,} k\in \{1,\dots, j_2\}. \end{equation}
Further, for such a point $\bsz=(z_1,z_2)\in I_{\bsj,\bsm}$ we have
\begin{equation} \label{z1} 2m_1+1-2^{j_1+1}z_1=1-t_{n-j_1}-2^{-1}t_{n-j_1-1}-\dots-2^{j_1-n+1}t_1 \end{equation}
and
\begin{equation} \label{z2} 2m_2+1-2^{j_2+1}z_2=1-b_{j_2+1}-2^{-1}b_{j_2+2}-\dots-2^{j_2-n+1}b_n. \end{equation}


There are several parallel tracks between the proofs in this section and the proofs in~\cite[Section 3]{Kritz},
where we computed the Haar coefficients for a simpler class of digital nets. \\
Let in the following $\mathcal{H}_j:=\{i\in\{1,\dots,j\}: a_i=0\}$ for $j\in\{1,\dots,n-1\}$. Then $h_n=|\mathcal{H}_{n-1}|$ is the parameter as defined in Theorem~\ref{theo1}.

\paragraph{Case 1: $\bsj\in\mathcal{J}_1:=\{(-1,-1)\}$}

\begin{proposition} \label{prop1}
  Let $\bsj\in \mathcal{J}_1$ and $\bsm\in \DD_{\bsj}$. Then we have
     $$ \mu_{\bsj,\bsm}=\frac{h_n+5}{2^{n+3}}+\frac{1}{2^{2n+2}}. $$
\end{proposition}

\begin{proof}
  By~\eqref{art1} we have
	\begin{align*}
	   \mu_{\bsj,\bsm}=& \frac{1}{2^n}\sum_{\bsz\in \cP_{\bsa}} (1-z_1)(1-z_2)-\frac14 \\
		  =&1-\frac{1}{2^n}\sum_{\bsz\in \cP_{\bsa}}z_1-\frac{1}{2^n}\sum_{\bsz\in \cP_{\bsa}}z_2+\frac{1}{2^n}\sum_{\bsz\in \cP_{\bsa}}z_1z_2-\frac14 \\
			=&-\frac14+\frac{1}{2^n}+\frac{1}{2^n}\sum_{\bsz\in \cP_{\bsa}}z_1z_2,
	\end{align*}
	where we regarded $\sum_{\bsz\in \cP_{\bsa}}z_1=\sum_{\bsz\in \cP_{\bsa}}z_2=\sum_{l=0}^{2^n-1}l/2^n=2^{n-1}-2^{-1}$ in the last step.
	It remains to evaluate $\sum_{\bsz\in \cP_{\bsa}}z_1z_2$. Using the representation of $\cP_{\bsa}$ in \eqref{darstPa}, we have
	\begin{align*}
	  \sum_{\bsz\in \cP}z_1z_2=& \sum_{t_1,\dots,t_n=0}^1 \left(\frac{t_n}{2}+\dots+\frac{t_{1}}{2^n}\right)\left(\frac{b_1}{2}+\dots+\frac{b_n}{2^n}\right) \\
		 =&\sum_{k=1}^n \sum_{t_1,\dots,t_n=0}^1 \frac{t_kb_k}{2^{n+1-k}2^k}+\sum_{\substack{k_1,k_2=1 \\ k_1 \neq k_2}}^n \sum_{t_1,\dots,t_n=0}^1 \frac{t_{k_1}b_{k_2}}{2^{n+1-k_1}2^{k_2}}=:S_1+S_2.
		\end{align*}
		Note that $b_k$ only depends on $t_k,t_{k+1},\ldots,t_n$ and $b_n=t_n$. We have
		\begin{align*}
		   S_1=& \frac{1}{2^{n+1}}\sum_{k=1}^n 2^{k-1} \sum_{t_k\dots,t_n=0}^1 t_kb_k=\frac{1}{2^{n+2}}2^n \sum_{t_n=0}^1 t_nb_n+\frac{1}{2^{n+2}}\sum_{k=1}^{n-1} 2^k \sum_{t_k\dots,t_n=0}^1 t_kb_k \\
			=& \frac14+\frac{1}{2^{n+2}}\sum_{k=1}^{n-1} 2^k \sum_{t_{k+1}\dots,t_n=0}^1 (1\oplus a_k(t_{k+1}\oplus \dots \oplus t_n)) \\
			=& \frac14+\frac{1}{2^{n+2}}\sum_{k=1}^{n-1} 2^k 2^{n-k-1}(2-a_k)=\frac14+\frac18 \left(n-1+\sum_{k=1}^{n-1}(1-a_k)\right)=\frac18(n+h_n+1).
		\end{align*}
		To compute $S_2$, assume first that $k_1<k_2$. Then
		\begin{align*}
	\sum_{t_1,\dots,t_n=0}^1 t_{k_1}b_{k_2}=& 2^{k_1-1} \sum_{t_{k_1},\dots,t_n=0}^1 t_{k_1}b_{k_2}=2^{k_1-1} \sum_{t_{k_1+1},\dots,t_n=0}^1 b_{k_2}\\
	=&2^{k_1-1}2^{k_2-k_1-1}\sum_{t_{k_2},\dots,t_n=0}^1 b_{k_2}=2^{k_1-1}2^{k_2-k_1-1}2^{n-k_2}=2^{n-2}.
		\end{align*}
		Similarly, we observe that we obtain the same result also for $k_1>k_2$ and hence
		$$ S_2=\frac{1}{2^{n+1}}\sum_{\substack{k_1,k_2=0 \\ k_1 \neq k_2}}^n 2^{k_1-k_2}2^{n-2}=\frac18 \sum_{\substack{k_1,k_2=0 \\ k_1 \neq k_2}}^n 2^{k_1-k_2}=\frac{1}{8}\left(-n+2^{n+1}-4+\frac{2}{2^n}\right). $$
		Now we put everything together to arrive at the claimed formula. \end{proof}

\paragraph{Case 2: $\bsj\in\mathcal{J}_2:=\{(-1,j_2): 0\leq j_2 \leq n-2\}$}
\begin{proposition} \label{prop2}
Let $\bsj=(-1,j_2)\in \mathcal{J}_2$ and $\bsm\in \DD_{\bsj}$. If $\mathcal{H}_{j_2}=\{1,\dots,j_2\}$, then
 $$ \mu_{\bsj,\bsm}=2^{-2n-2j_2-4}\left(-2^{2j_2+2}(a_{j_2+1}-1)+2^{n+j_2}(a_{j_2+1}a_{j_2+2}-2)+2^{2n+2}\sum_{k=1}^{j_2}\frac{s_k}{2^{n+1-k}}\right), $$
where the latter sum is zero for $j_2=0$.
Otherwise, let $w\in\{1,\dots,j_2\}$ be the greatest index with $a_w=1$. If $a_{j_2+1}=0$, then
\begin{align*} \mu_{\bsj,\bsm}=&2^{-2n-2}-2^{-n-j_2-3}+2^{-n-2j_2+w-5}+2^{-2j_2-2}\varepsilon \\ &+2^{-2n-j_2+w-4}a_{j_2+2}(1-2(s_{w}\oplus \dots \oplus s_{j_2})). \end{align*}
If $a_{j_2+1}=1$, then
\begin{align*} \mu_{\bsj,\bsm}=&-2^{-n-j_2-3}+2^{-j_2+w-2n-3}+2^{-2j_2-n+w-4}+2^{-2j_2-2}\varepsilon \\ &-2^{-2n-j_2+w-2}(s_{w}\oplus \dots \oplus s_{j_2})+2^{-n-j_2-4}a_{j_2+2}. \end{align*}
In the latter two expressions, we put $\varepsilon=\sum_{\substack{k=1 \\k\neq w} }^{j_2}\frac{t_k(m_2)}{2^{n+1-k}}$, where the values $t_k(m_2)$ depend only on $m_2$ and are either 0 or 1. Hence, in any case we have $|\mu_{\bsj,\bsm}|\lesssim 2^{-n-j_2}$.
\end{proposition} 

\begin{proof}
We only show the case where $j_2\geq 1$ and $\cH_{j_2} \neq \{1,\dots,j_2\}$, since the other case is similar but easier.
  Let $w\in\{1,\dots,j_2\}$ be the greatest index with $a_w=1$. By~\eqref{art3}, we need to evaluate the sum
	   $$ \sum_{\bsz\in \cP_{\bsa}\cap I_{\bsj,\bsm}}(1-z_1)(1-|2m_2+1-2^{j_2+1}z_2|). $$
	By~\eqref{cond}, the condition $\bsz\in \cP_{\bsa}\cap I_{\bsj,\bsm}$ yields the identities $b_k=s_k$ for all $k\in \{1,\dots,j_2\}$, which lead to 
 $t_k=s_k$ for all $k\in\{1,\dots,j_2\}$ such that $a_k=0$. Assume
	that $$ \{k\in \{1,\dots,j_2\}: a_k=1\}=\{k_1,\dots,k_l\} $$ for some $l\in\{1,\dots j_2\}$, where $k_1<k_2<\dots<k_l$ and $k_l=w$.
	We have $t_{k_i}=s_{k_i}\oplus s_{k_i+1}\oplus \dots \oplus s_{k_{i+1}}$ for all $i\in\{1,\dots,l-1\}$ and
	$t_w=s_w \oplus \dots \oplus s_{j_2} \oplus t_{j_2+1}\oplus \dots \oplus t_n$. Hence, we can write
	$$ 1-z_1=1-u-\frac{t_{j_2+1}}{2^{n-j_2}}-\frac{s_w \oplus \dots \oplus s_{j_2} \oplus t_{j_2+1}\oplus \dots \oplus t_n}{2^{n+1-w}}-\varepsilon, $$
	where $u=2^{-1}t_n+\dots+2^{-(n-j_2-1)}t_{j_2+2}$ and $$\varepsilon=\varepsilon(m_2)=\sum_{\substack{k=1 \\k\neq w} }^{j_2}\frac{t_k(m_2)}{2^{n+1-k}}.$$
	For the expression $1-|2m_2+1-2^{j_2+1}z_2|$ we find by~\eqref{z2}
	$$ 1-|2m_2+1-2^{j_2+1}z_2|=1-|1-t_{j_2+1}\oplus a_{j_2+1}(t_{j_2+2}\oplus \dots \oplus t_n)-v|, $$
	where $v=v(t_{j_2+2},\dots,t_n)=2^{-1}b_{j_2+2}+\dots+2^{-(n-j_2-1)}b_n.$ With these observations, we find (writing $T_j=t_j\oplus \dots\oplus t_n$ for $1\leq j \leq n-1$ and $t_w(t_{j_2+1})=s_{w}\oplus\dots\oplus s_{j_2}\oplus t_{j_2+1} \oplus T_{j_2+2}$)
	\begin{align*}
	  \sum_{\bsz\in \cP_{\bsa}\cap I_{\bsj,\bsm}}&(1-z_1)(1-|2m_2+1-2^{j_2+1}z_2|) \\
		    =& \sum_{t_{j_2+1}, \dots , t_n=0}^{1}\left(1-u-\frac{t_{j_2+1}}{2^{n-j_2}}-\frac{t_w(t_{j_2+1})}{2^{n+1-w}}-\varepsilon\right) \\
				  &\times\left(1-|1-t_{j_2+1}\oplus a_{j_2+1}T_{j_2+2}-v|\right) \\
				=&\sum_{t_{j_2+2}, \dots , t_n=0}^{1} \bigg\{ \left(1-u-\frac{a_{j_2+1}T_{j_2+2}}{2^{n-j_2}}-\frac{t_w(a_{j_2+1}T_{j_2+1})}{2^{n+1-w}}-\varepsilon\right)v \\
				 &+ \left(1-u-\frac{a_{j_2+1}T_{j_2+2}\oplus 1}{2^{n-j_2}}-\frac{t_w(a_{j_2+1}T_{j_2+1}\oplus 1)}{2^{n+1-w}}-\varepsilon\right)(1-v) \bigg\}\\
			=& \sum_{t_{j_2+2}, \dots , t_n=0}^{1} 2^{-n-1}\bigg(-2^{j_2+1}-2^w+2^{n+1}-2^{n+1}\varepsilon+2^wt_w(a_{j_2+1}T_{j_2+1})-2^{n+1}u \\ &+2^{j_2+1}v+2^w v 
			-2^{w+1}t_w(a_{j_2+1}T_{j_2+1})v-2^{j_2+1}(a_{j_2+1}T_{j_2+2})(2v-1)\bigg).
	\end{align*}
	Let first $a_{j_2+1}=1$ and hence $t_w(a_{j_2+1}T_{j_2+2})=t_w(T_{j_2+2})=s_{w}\oplus\dots\oplus s_{j_2}$ does not depend on t $t_i$. 
	Since
	$$ \sum_{t_{j_2+2}, \dots , t_n=0}^{1} u=\sum_{t_{j_2+2}, \dots , t_n=0}^{1} v=\sum_{l=0}^{2^{n-j_2-1}-1}\frac{l}{2^{n-j_2+1}}=2^{n-j_2-2}-\frac12, $$
	we obtain
	\begin{align*}
	  \sum_{\bsz\in \cP_{\bsa}\cap I_{\bsj,\bsm}}&(1-z_1)(1-|2m_2+1-2^{j_2+1}z_2|) \\
		 =& 2^{-n-1}\bigg((-2^{j_2+1}-2^w+2^{n+1}-2^{n+1}\varepsilon+2^wt_w(T_{j_2+2}))2^{n-j_2-1} \\
		 &+(2^w+2^{j_2+1}-2^{n+1}-2^{w+1}t_w(T_{j_2+2}))\left(2^{n-j_2-2}-\frac12\right) \\
		&-2^{j_2+1}\sum_{t_{j_2+2}, \dots , t_n=0}^{1}T_{j_2+2}(2v-1)\bigg).
		\end{align*}
	We analyze the last expression. We find
	\begin{align*} \sum_{t_{j_2+2}, \dots , t_n=0}^{1}& T_{j_2+2}(2v-1) \\ =& 2\sum_{t_{j_2+2}, \dots , t_n=0}^{1}T_{j_2+2}v-\sum_{t_{j_2+2}, \dots , t_n=0}^{1}T_{j_2+2}=2\sum_{t_{j_2+2}, \dots , t_n=0}^{1}T_{j_2+1}v-2^{n-j_2-2},\end{align*}
	where
	\begin{align*}
	  \sum_{t_{j_2+2}, \dots , t_n=0}^{1}T_{j_2+2}v=&\sum_{t_{j_2+2}, \dots , t_n=0}^{1} (t_{j_2+2}\oplus T_{j_2+3})\left(\frac{t_{j_2+2}\oplus a_{j_2+2} T_{j_2+3}}{2}+\frac{b_{j_2+3}}{4}+\cdots+\frac{b_n}{2^{n-j_1-1}}\right) \\
		=& \sum_{t_{j_2+3}, \dots , t_n=0}^{1} \left(\frac{(T_{j_2+3}\oplus 1)\oplus a_{j_2+2} T_{j_2+3}}{2}+\frac{b_{j_2+3}}{4}+\cdots+\frac{b_n}{2^{n-j_2-1}}\right) \\
		=& \sum_{t_{j_2+3}, \dots , t_n=0}^{1} \frac{1\oplus (1-a_{j_2+2})T_{j_2+3}}{2}+\sum_{l=0}^{2^{n-j_2-2}-1}\frac{l}{2^{n-j_2-1}} \\
		=&\frac12 \sum_{t_{j_2+3}, \dots , t_n=0}^{1} \left(1-(1-a_{j_2+2})T_{j_2+3}\right)+2^{n-j_2-4}-\frac14 \\
		=&\frac12 \left(2^{n-j_2-2}-(1-a_{j_2+2})2^{n-j_2-3}\right)+2^{n-j_2-4}-\frac14 \\
		=&2^{n-j_2-4}(1+a_{j_2+2})+2^{n-j_2-4}-\frac14.
	\end{align*}
	We put everything together and apply~\eqref{art3} to find the result for $a_{j_2+1}=1$. \\
	Now assume that $a_{j_2+1}=0$. Then $t_w(a_{j_2+1}T_{j_2+2})=t_w(0)=s_w\oplus\dots\oplus s_{j_2} \oplus T_{j_2+2}$. Hence we have
		\begin{align*}
	  \sum_{\bsz\in \cP_{\bsa}\cap I_{\bsj,\bsm}}&(1-z_1)(1-|2m_2+1-2^{j_2+1}z_2|) \\
		  =& 2^{-n-1}\bigg( (-2^{j_1+1}+2^{n+1}-2^w-2^{n+1}\varepsilon)2^{n-j_2-1}+(2^{j_2+1}-2^{n+1})(2^{n-j_2-2}-\frac12) \\
			      &+ 2^w \cdot 2^{n-j_2-2}-2^{w+1}\sum_{t_{j_2+2},\dots,t_n=0}^{1} vt_w(0)\bigg).
		\end{align*}
		We considered $\sum_{t_{j_2+2},\dots,t_n=0}^{1}t_w(0)=2^{n-j_2-2}.$ It remains to evaluate $\sum_{t_{j_2+2},\dots,t_n=0}^{1} vt_w(0).$ We find
		\begin{eqnarray*}
		    \lefteqn{\sum_{t_{j_2+2},\dots,t_n=0}^{1} (s_w\oplus\dots\oplus s_{j_2}\oplus t_{j_2+2}\oplus T_{j_2+3})\left(\frac{t_{j_2+2}\oplus a_{j_2+2}T_{j_2+3}}{2}+\frac{b_{j_2+2}}{4}+\dots+\frac{b_n}{2^{n-j_2-1}}\right)} \\
				  &=& \sum_{t_{j_2+3},\dots,t_n=0}^{1} \left(\frac{(s_w\oplus\dots\oplus s_{j_2}\oplus T_{j_2+3} \oplus 1)\oplus a_{j_2+1}T_{j_2+3}}{2}+\frac{b_{j_2+2}}{4}+\dots+\frac{b_n}{2^{n-j_2-1}}\right) \\
					&=& \frac12 \sum_{t_{j_2+3},\dots,t_n=0}^{1} (1-a_{j_2+2})T_{j_2+3}\oplus s_w\oplus\dots\oplus s_{j_2} \oplus 1 +\sum_{l=0}^{2^{n-j_2-2}-1}\frac{l}{2^{n-j_2-1}} \\
					&=& 2^{n-j_2-4}(1+a_{j_2+2}(1-2(s_w\oplus\dots\oplus s_{j_2})) +2^{n-j_2-4}-\frac14.
		\end{eqnarray*}
		Again, we put everything together and apply~\eqref{art3} to find the result for $a_{j_2+1}=0$. 
\end{proof}

\paragraph{Case 3: $\bsj\in\mathcal{J}_3:=\{(k,-1): k\geq n\}\cup\{(-1,k): k\geq n\}$}

\begin{proposition}
  Let $\bsj\in \mathcal{J}_3$ and $\bsm\in \DD_{\bsj}$. Then we have
     $$ \mu_{\bsj,\bsm}=\frac{1}{2^{2k+3}}. $$
\end{proposition}

\begin{proof}
 This claim follows from~\eqref{art2} and~\eqref{art3} together with the fact that no point of $\cP_{\bsa}$ is contained 
in the interior of $I_{\bsj,\bsm}$ if $j_1\geq n$ or $j_2\geq n$. Hence, only the linear part of $\Delta_{\cP_{\bsa}}$ 
contributes to the Haar coefficients in this case.
\end{proof}

\paragraph{Case 4: $\bsj\in\mathcal{J}_4:=\{(0,-1)\}$}

\begin{proposition} \label{prop5}
  Let $\bsj\in \mathcal{J}_4$ and $\bsm\in \DD_{\bsj}$. Then we have
     $$ \mu_{\bsj,\bsm}=-\frac{1}{2^{n+3}}+\frac{1}{2^{2n+2}}. $$
\end{proposition}

\begin{proof}
For $\bsz=(z_1,z_2)\in  \cP_{\bsa}\cap I_{\bsj,\bsm}=\cP_{\bsa} $ we have $1-z_2=1-\frac{b_1}{2}-\dots-\frac{b_n}{2^n}$ and
	  $$ 1-|2m_1+1-2z_1|=1-\left|1-t_n-\frac{t_{n-1}}{2}-\dots-\frac{t_1}{2^{n-1}}\right| $$ by~\eqref{z1}.
		We therefore find, after summation over $t_n$,
		\begin{align*}
		    \sum_{\bsz \in \cP_{\bsa}\cap I_{\bsj,\bsm}}&(1-|2m_1+1-2z_1|)(1-z_2)  \\
				=&\sum_{t_1,\dots,t_n=0}^1 \left(1-\left|1-t_n-\frac{t_{n-1}}{2}-\dots-\frac{t_1}{2^{n-1}}\right|\right)\left(1-\frac{b_1(t_n)}{2}-\dots-\frac{b_n(t_n)}{2^n}\right) \\
				=&\sum_{t_1,\dots,t_{n-1}=0}^1 \left(u(1-v(0))+(1-u)\left(1-v(1)-\frac{1}{2^n}\right)\right) \\
				=& \sum_{t_1,\dots,t_{n-1}=0}^1 \left(1-\frac{1}{2^n}-v(1)+\frac{1}{2^n}u+uv(1)-uv(0)\right) \\
				=& 2^{n-1}\left(1-\frac{1}{2^n}\right)+\left(\frac{1}{2^n}-1\right)(2^{n-2}-2^{-1})+\sum_{t_1,\dots,t_{n-1}=0}^1uv(1)-\sum_{t_1,\dots,t_{n-1}=0}^1uv(0).
		\end{align*}
		Here we use the short-hands $u=2^{-1}t_{n-1}+\dots+2^{-n+1}t_1$ and $v(t_n)=2^{-1}b_1(t_n)+\dots+2^{-n+1}b_{n-1}(t_n)$ and the fact
		that $\sum_{t_1,\dots,t_{n-1}=0}^1 u= \sum_{t_1,\dots,t_{n-1}=0}^1 v(1)=2^{n-2}-2^{-1}$. It is not difficult to observe
		that $\sum_{t_1,\dots,t_{n-1}=0}^1 uv(0)= \sum_{t_1,\dots,t_{n-1}=0}^1 uv(1)$; hence
		$$ \sum_{\bsz \in \cP_{\bsa}}(1-|2m_1+1-2z_1|)(1-z_2)=\frac14+2^{n-2}-\frac{1}{2^{n+1}}. $$
		The rest follows with~\eqref{art2}.
\end{proof}

For the following two propositions, we use the shorthand $R=r_1\oplus \dots \oplus r_{j_1}$.

\paragraph{Case 5: $\bsj\in\mathcal{J}_5:=\{(j_1,-1): 1\leq j_1 \leq n-2 \}$}

\begin{proposition}
  Let $\bsj\in \mathcal{J}_5$ and $\bsm\in \DD_{\bsj}$. Then we have
     $$ \mu_{\bsj,\bsm}=2^{-2n-2}-2^{-n-j_1-3}+2^{-2j_1-2}\varepsilon-2^{-2n-1}R-2^{-n-j_1-3}a_{n-j_1-1}(1-2R), $$
     where \begin{equation} \label{abcd} \varepsilon=\varepsilon(m_1)=\frac{r_1}{2^n}+\sum_{k=2}^{j_1}\frac{r_k \oplus a_{n+1-k}(r_{k-1}\oplus\dots\oplus r_1)}{2^{n+1-k}}.\end{equation}
		Hence, we have
\end{proposition}

\begin{proof}
By~\eqref{art2}, we need to evaluate the sum
 $$\sum_{\bsz\in \cP_{\bsa}\cap I_{\bsj,\bsm}} (1-|2m_1+1-2^{j_1+1}z_1|)(1-z_2).$$
The condition $\bsz\in \cP_{\bsa}\cap I_{\bsj,\bsm}$ forces $t_n=r_1,\dots t_{n+1-j_1}=r_{j_1}$ and therefore
\begin{align*}
  1-z_2=& 1-\frac{b_1}{2}-\dots-\frac{b_n}{2^n}=1-v(t_{n-j_1})-\frac{t_{n-j_1}\oplus a_{n-j_1} R}{2^{n-j_1-1}}-\varepsilon,
\end{align*}
where 
\begin{align*} v(t_{n-j_1})&=\frac{b_1}{2}+\dots+\frac{b_{n-j_1-1}}{2^{n-j_1-1}} \\ &=\frac{t_1\oplus a_1(t_2\oplus\dots\oplus t_{n-j_1}\oplus R)}{2}+\dots+\frac{t_{n-j_1-1}\oplus a_{n-j_1-1}(t_{n-j_1}\oplus R)}{2^{n-j_1-1}} \end{align*}
and $\varepsilon$ as in~\eqref{abcd}.
Further, by~\eqref{z1} we write $2m_1+1-2^{j_1+1}z_1=1-t_{n-j_1}-u$, where $u=2^{-1}t_{n-j_1-1}+\dots+2^{j_1-n+1}t_1$. Then
\begin{align*}
    \sum_{\bsz\in \cP_{\bsa}\cap I_{\bsj,\bsm}} &(1-|2m_1+1-2^{j_1+1}z_1|)(1-z_2) \\
		  =& \sum_{t_1,\dots,t_{n-j_1}=0}^{1}\left(1-v(t_{n-j_1})-\frac{t_{n-j_1}\oplus a_{n-j_1} R}{2^{n-j_1-1}}-\varepsilon\right)(1-|1-t_{n-j_1}-u|) \\
			=&  \sum_{t_1,\dots,t_{n-j_1-1}=0}^{1}  \Bigg\{\left(1-v(0)-\frac{a_{n-j_1} R}{2^{n-j_1-1}}-\varepsilon\right)u\\&+\left(1-v(1)-\frac{1\oplus a_{n-j_1} R}{2^{n-j_1-1}}-\varepsilon\right)(1-u)\Bigg\}\\
			=& \sum_{t_1,\dots,t_{n-j_1-1}=0}^{1} \{ 1-2^{j_1-n}-\varepsilon+2^{j_1-n}a_{n-j_1}R+2^{j_1-n}u-v(1)\\ &-2^{1+j_1-n}a_{n-j_1}Ru+uv(1)-uv(0) \} \\
			=& 2^{n-j_1-1}( 1-2^{j_1-n}-\varepsilon+2^{j_1-n}a_{n-j_1}R) \\&+\left(2^{n-j_1-2}-2^{-1}\right)(2^{j_1-n}-1-2^{1+j_1-n}a_{n-j_1}R) 
			+\sum_{t_1,\dots,t_{n-j_1-1}=0}^{1}(uv(1)-uv(0)).
\end{align*}
We understand $b_1,\dots,b_{n-j_1-1}$ as functions of $t_{n-j_1}$ and have
\begin{align*}
  \sum_{t_1,\dots,t_{n-j_1-1}=0}^{1} uv(0)=& \sum_{t_1,\dots,t_{n-j_1-1}=0}^{1}\left(\frac{t_{n-j_1-1}}{2}+\dots+\frac{t_1}{2^{n-j_1-1}}\right)\left(\frac{b_1(0)}{2}+\dots+\frac{b_{n-j_1-1}(0)}{2^{n-j_1-1}}\right) \\
	 =& \sum_{t_1,\dots,t_{n-j_1-1}=0}^{1} \left(\sum_{k=1}^{n-j_1-1} \frac{t_kb_k(0)}{2^{n-j_1-k}2^k}+\sum_{\substack{k_1,k_2=0 \\ k_1 \neq k_2}}^{n-j_1-1} \frac{t_{k_1}b_{k_2}(0)}{2^{n-j_1-k_1}2^{k_2}}\right).
\end{align*}
The first sum simplifies to
\begin{align*}
  \sum_{k=1}^{n-j_1-1}& 2^{k-1}\sum_{t_k,\dots,t_{n-j_1-1}=0}^{1}\frac{t_kb_k(0)}{2^{n-j_1-k}2^k} \\
	   =& \frac{1}{2^{n-j_1}}\sum_{k=1}^{n-j_1-2}2^{k-1}\sum_{t_k,\dots,t_{n-j_1-1}=0}^{1}t_k(t_k\oplus a_k(t_{k+1}\oplus t_{n-j_1-1}\oplus R)) \\
		 &+\frac{1}{2^{n-j_1}}2^{n-j_1-2}\sum_{t_{n-j_1-1}=0}^{1}t_{n-j_1-1}(t_{n-j_1-1}\oplus a_{n-j_1-1}R) \\
		 =& \frac{1}{2^{n-j_1}}\sum_{k=1}^{n-j_1-2}2^{k-1}\sum_{t_{k+1},\dots,t_{n-j_1-1}=0}^{1}(1\oplus a_k(t_{k+1}\oplus t_{n-j_1-1}\oplus R)) \\
		 &+\frac{1}{4}(1 \oplus a_{n-j_1-1}(R\oplus 1)) \\
		=& \frac{1}{2^{n-j_1}}\sum_{k=1}^{n-j_1-2}2^{k-1}2^{n-j_1-k-2}(2-a_k)+\frac14 (1 \oplus a_{n-j_1-1}(R\oplus 1)) \\
		=& \frac{1}{8}\sum_{k=1}^{n-j_1-2}(2-a_k)+\frac14 (1 \oplus a_{n-j_1-1}(R\oplus 1)).
\end{align*}
Basically by the same arguments as in the proof of Proposition~\ref{prop1} we also find
\begin{align*}
  \sum_{t_1,\dots,t_{n-j_1-1}=0}^{1}\sum_{\substack{k_1,k_2=0 \\ k_1 \neq k_2}}^{n-j_1-1} \frac{t_{k_1}b_{k_2}}{2^{n-j_1-k_1}2^{k_2}}=\frac18 \sum_{\substack{k_1,k_2=0 \\ k_1 \neq k_2}}^1 2^{k_1-k_2}.
\end{align*}
Hence, we obtain
$$ \sum_{t_1,\dots,t_{n-j_1-1}=0}^{1} uv(0)=\frac{1}{8}\sum_{k=1}^{n-j_1-2}(2-a_k)+\frac14(1 \oplus a_{n-j_1-1}(R\oplus 1))+\frac18 \sum_{\substack{k_1,k_2=0 \\ k_1 \neq k_2}}^1 2^{k_1-k_2}. $$
We can evaluate $\sum_{t_1,\dots,t_{n-j_1-1}=0}^{1} uv(1)$ in almost the same way; the result is
$$ \sum_{t_1,\dots,t_{n-j_1-1}=0}^{1} uv(1)=\frac{1}{8}\sum_{k=1}^{n-j_1-2}(2-a_k)+\frac14 (1 \oplus a_{n-j_1-1}R)+\frac18 \sum_{\substack{k_1,k_2=0 \\ k_1 \neq k_2}}^1 2^{k_1-k_2}. $$
Hence the difference of these two expressions is given by
$$ \sum_{t_1,\dots,t_{n-j_1-1}=0}^{1} uv(1)-\sum_{t_1,\dots,t_{n-j_1-1}=0}^{1} uv(0)=\frac14 a_{n-j_1-1}(2R-1). $$
Now we put everything together and use~\eqref{art2} to find the claimed result on the Haar coefficients.
\end{proof}

\paragraph{Case 6: $\bsj\in\mathcal{J}_6:=\{(j_1,j_2): j_1+j_2 \leq n-3 \}$}

\begin{proposition}
Let $\bsj\in \mathcal{J}_6$ and $\bsm\in \DD_{\bsj}$. If $\mathcal{H}_{j_2}=\{1,\dots,j_2\}$ or if $j_2=0$, then we have
 $$ \mu_{\bsj,\bsm}=2^{-2n-2}(1-2a_{n-j_1}R)(1-a_{j_2+1}). $$
Otherwise, let $w\in\{1,\dots,j_2\}$ be the greatest index with $a_w=1$. If $a_{j_2+1}=0$, then
$$ \mu_{\bsj,\bsm}=2^{-2n-2}(1-2a_{n-j_1}R). $$
If $a_{j_2+1}=1$, then
$$ \mu_{\bsj,\bsm}=-2^{-2n-j_2+w-3}(1-2a_{n-j_1}R)(1-2(s_w\oplus\dots\oplus s_{j_2})). $$
Note that for $j_1=0$ we set $a_{n-j_1}R=0$ in all these formulas. Hence, in any case we have $|\mu_{\bsj,\bsm}|\lr 2^{-2n}$
\end{proposition} 

\begin{proof}
  The proof is similar in all cases; hence we only treat the most complicated case where $j_2\geq 1$ and $\cH_{j_2} \neq \{1,\dots,j_2\}$.
	By~\eqref{art4}, we need to study the sum
	$$  \sum_{\bsz\in \cP_{\bsa}\cap I_{\bsj,\bsm}} (1-|2m_1+1-2^{j_1+1}z_1|)(1-|2m_2+1-2^{j_2+1}z_2|), $$
	where the condition $\bsz\in \cP_{\bsa}\cap I_{\bsj,\bsm}$ forces $t_{n+1-k}=r_k$ for all $k\in\{1,\dots,j_1\}$
	as well as $b_k=s_k$ for all $k\in\{1,\dots,j_2\}$. 	
	We have already seen in the proof of Proposition 2 that the latter equalities allow us to
	express the digits $t_k$ by the digits $s_1,\dots,s_{j_2}$ of $m_2$ for all $k\in\{1,\dots,j_2\}\setminus\{w\}$.
	We also have $t_w=s_w\oplus\dots\oplus s_{j_2}\oplus t_{j_2+1}\oplus\dots\oplus t_n$. With~\eqref{z1}, these observations lead to
	$$ 2m_1+1-2^{j_1+1}z_1=1-t_{n-j_1}-u-2^{j_1+j_2-n+1}t_{j_2+1}-2^{j_1+w-1}t_w-\varepsilon_2(m_2), $$
	where $u=2^{-1}t_{n-j_1-1}+\dots+2^{j_1+j_2-n+2}t_{j_2+2}$ and 
	$ \varepsilon_2 $ is determined by $m_2$.
	Further, we write with~\eqref{z2}
	$$ 2m_2+1-2^{j_2+1}z_2=1-b_{j_2+1}-v-2^{j_1+j_2-n+1}b_{n-j_1}-\varepsilon_1(m_1), $$
	where $v=v(t_{n-j_1})=2^{-1}b_{j_1+2}+\dots+2^{j_1+j_2-n+2}b_{n-j_1-1}$ and $\varepsilon_1$ is obviously determined by $m_1$. Hence, we have
	\begin{align*}
	   \sum_{\bsz\in \cP_{\bsa}\cap I_{\bsj,\bsm}} &(1-|2m_1+1-2^{j_1+1}z_1|)(1-|2m_2+1-2^{j_2+1}z_2|) \\
		=& \sum_{t_{j_2+1},\dots,t_{n-j_1}=0}^1 \left(1-|1-t_{n-j_1}-u-2^{j_1+j_2-n+1}t_{j_2+1}-2^{j_1+w-1}t_w-\varepsilon_2(m_2)|\right) \\
		&\times \bigg(1-|1-t_{j_2+1}\oplus a_{j_2+1}(t_{j_2+2}\oplus \dots \oplus t_{n-j_1}\oplus R)-v(t_{n-j_1})\\ &-2^{j_1+j_2-n+1}(t_{n-j_1}\oplus a_{n-j_1}R)-\varepsilon_1)|\bigg).
	\end{align*}
	Recall we may write $t_w=s_w\oplus \dots \oplus s_{j_2} \oplus t_{j_2+1}\oplus t_{j_2+2}\oplus\dots\oplus t_{n-j_1-1} \oplus t_{n-j_1}\oplus R$. We stress the
	dependence of $t_w$ on $t_{j_2+1}\oplus t_{n-j_1}$ by writing $t_w(t_{j_2+1}\oplus t_{n-j_1})$. If $a_{j_2+1}=0$, then we
	 obtain after summation over $t_{j_2+1}$ and $t_{n-j_1}$
	\begin{align*}
	   \sum_{\bsz\in \cP_{\bsa}\cap I_{\bsj,\bsm}} &(1-|2m_1+1-2^{j_1+1}z_1|)(1-|2m_2+1-2^{j_2+1}z_2|) \\
		  =& \sum_{t_{j_2+2},\dots,t_{n-j_1-1}=0}^{1} \bigg\{
			    (u+2^{j_1+w-1}t_w(0)+\varepsilon_2)(v(0)+2^{j_1+j_2-n+1}a_{n-j_1}R+\varepsilon_1)\\
					&+  (u+2^{j_1+j_2-n+1}+2^{j_1+w-1}t_w(1)+\varepsilon_2)(1-v(0)-2^{j_1+j_2-n+1}a_{n-j_1}R-\varepsilon_1) \\
					&+ (1-u-2^{j_1+w-1}t_w(1)-\varepsilon_2)(v(0)+2^{j_1+j_2-n+1}(a_{n-j_1}R\oplus 1)+\varepsilon_1)\\
					&+ (1-u-2^{j_1+j_2-n+1}-2^{j_1+w-1}t_w(0)-\varepsilon_2)(1-v(1) \\ &-2^{j_1+j_2-n+1}(a_{n-j_1}R\oplus 1)-\varepsilon_1)
			\bigg\} \\
			=& \sum_{t_{j_2+2},\dots,t_{n-j_1-1}=0}^{1} \bigg\{1+2^{2(n+j_1+j_2+1)}+2^{j_1+w-1}-2^{2j_1+j_2-n+w}-2^{2j_1+2j_2-2n+3}a_{n-j_1}R \\
			&+(2^{1+2j_1+j_2-n+w}-2^{j_1+w})t_w(0) +2^{j_1+w}(2t_w(0)-1)+2^{n+j_1+j_2+1}(v(1)-v(0)) \\
			&-2^{w+j_1-1}(v(1)+v(0))+2^{j_1+w}(t_w(0)v(0)+t_w(0)v(1)).\bigg\}
	\end{align*}
	We regarded $t_w(1)=1-t_w(0)$. By standard argumentation, we find 
	$$ \sum_{t_{j_2+2},\dots,t_{n-j_1-1}=0}^{1}v(0)=\sum_{t_{j_2+2},\dots,t_{n-j_1-1}=0}^{1}v(1)=2^{n-j_1-j_2-3}-\frac12 $$
	and
	$$ \sum_{t_{j_2+2},\dots,t_{n-j_1-1}=0}^{1} t_w(0)=\sum_{t_{j_2+2},\dots,t_{n-j_1-2}=0}^{1}1=2^{n-j_1-j_2-3}. $$
	We use the short-hand $T=t_{j_2+3}\oplus\dots\oplus t_{n-j_1-1}$, which allows us to write
	\begin{align*}
	   \sum_{t_{j_2+2},\dots,t_{n-j_1-1}=0}^{1} &t_w(0)v(0)=\sum_{t_{j_2+2},\dots,t_{n-j_1-1}=0}^{1} (s_w\oplus\dots \oplus s_{j_2}\oplus t_{j_2+2} \oplus \dots\oplus t_{n-j_1-1}\oplus R) \\
		  &\times \bigg( \frac{t_{j_2+2}\oplus a_{j_2+2}(t_{j_2+3}\oplus\dots\oplus t_{n-j_1-1}\oplus R)}{2} \\
			&+\frac{t_{j_2+3}\oplus a_{j_2+3}(t_{j_2+4}\oplus\dots\oplus t_{n-j_1-1}\oplus R)}{4}+\dots+\frac{t_{n-j_1-1}\oplus a_{n-j_1-1}R}{2^{n-j_1-j_2-2}}\bigg) \\
			=& \sum_{t_{j_2+3},\dots,t_{n-j_1-1}=0}^{1}\frac12 (s_w\oplus\dots \oplus s_{j_2}\oplus T \oplus R \oplus 1 \oplus a_{j_2+2}(T\oplus R)) \\&+\sum_{l=0}^{2^{n-j_1-j_2-3}-1}\frac{l}{2^{n-j_1-j_2-2}} \\
			=& \sum_{t_{j_2+3},\dots,t_{n-j_1-1}=0}^{1}\frac12 (s_w\oplus\dots \oplus s_{j_2}\oplus 1 \oplus (1-a_{j_2+2})(T\oplus R))\\ &+2^{n-j_1-j_2-5}-\frac14.
	\end{align*}
	Similarly, we can show
	\begin{align*}
	   \sum_{t_{j_2+2},\dots,t_{n-j_1-1}=0}^{1} t_w(0)v(1)=&\sum_{t_{j_2+3},\dots,t_{n-j_1-1}=0}^{1}\frac12 (s_w\oplus\dots \oplus s_{j_2} \oplus (1-a_{j_2+2})(T\oplus R \oplus 1))
		\\ &+2^{n-j_1-j_2-5}-\frac14
	\end{align*}
	and therefore
	$$  \sum_{t_{j_2+2},\dots,t_{n-j_1-1}=0}^{1} t_w(0)(v(0)+v(1))=2^{n-j_1-j_2-4}+2^{n-j_1-j_2-5}-\frac14, $$
	a fact which can be found by distinguishing the cases $a_{j_2+1}=0$ and $a_{j_2+1}=1$.
	We put everything together and obtain 
	\begin{align*} \sum_{\bsz\in \cP_{\bsa}\cap I_{\bsj,\bsm}} &(1-|2m_1+1-2^{j_1+1}z_1|)(1-|2m_2+1-2^{j_2+1}z_2|) \\ &=2^{j_1+j_2-n}+2^{n-j_1-j_2-2}-2^{-n+j_1+j_2+1}a_{n-j_1}R, \end{align*}
	which leads to the claimed result for $a_{j_2+1}=0$ via~\eqref{art4}.\\
	Now assume that $a_{j_2+1}=1$. In this case, it is more convenient to consider $t_w$ as a function
	of $t_{j_2+1}\oplus\dots\oplus t_{n-j_1}\oplus R$. We obtain  after summation over $t_{j_2+1}$ and $t_{n-j_1}$
	\begin{align*}
	   \sum_{\bsz\in \cP_{\bsa}\cap I_{\bsj,\bsm}} &(1-|2m_1+1-2^{j_1+1}z_1|)(1-|2m_2+1-2^{j_2+1}z_2|) \\
		  =& \sum_{t_{j_2+2},\dots,t_{n-j_1-1}=0}^{1} \bigg\{
			    (u+2^{j_1+j_2-n+1}(T\oplus R)+2^{j_1+w-1}t_w(0)+\varepsilon_2)\\&\times(v(0) +2^{j_1+j_2-n+1}a_{n-j_1}R+\varepsilon_1)\\
					&+  (u+2^{j_1+j_2-n+1}(T\oplus R\oplus 1)+2^{j_1+w-1}t_w(1)+\varepsilon_2)\\&\times(1-v(0)-2^{j_1+j_2-n+1}a_{n-j_1}R-\varepsilon_1) \\
					&+ (1-u-2^{j_1+j_2-n+1}(T\oplus R\oplus 1)-2^{j_1+w-1}t_w(0)-\varepsilon_2)\\&\times(v(1)+2^{j_1+j_2-n+1}(a_{n-j_1}R\oplus 1)+\varepsilon_1)\\
					&+ (1-u-2^{j_1+j_2-n+1}(T\oplus R)-2^{j_1+w-1}t_w(1)-\varepsilon_2)\\&\times(1-v(1)-2^{j_1+j_2-n+1}(a_{n-j_1}R\oplus 1)-\varepsilon_1)
			\bigg\} \\
			=& \sum_{t_{j_2+2},\dots,t_{n-j_1-1}=0}^{1} 2^{-2n}\bigg\{2^{n+j_1+j_2+1}+2^{2j_1+j_2+w+1}(1-2t_w(0)+2a_{n-j_1}R(2t_w(0)-1)) \\
			  &-2^{2(j_1+j_2+1)}+2^{2n}+(2^{2j_1+2j_2+3}-2^{n+j_1+j_2+2})(T\oplus R) \\ &+2^{n+j_1+j_2+2}\varepsilon_1(2(T\oplus R)-1)\\
				&-2^{n+j_1+j_2+1}(v(0)+v(1))+2^{n+j_1+j_2+1}(2t_w(0)-1)(v(1)-v(0))\\&+2^{n+j_1+w}(v(0)+v(1)) +2^{n+j_1+j_2+2}(T\oplus R)(v(1)+v(0)).\bigg\}
	\end{align*}
	Again, we used $t_w(1)=1-t_w(0)$. Note that $t_w(0)=s_w\oplus\dots\oplus s_{j_2}$ is independent of the
	digits $t_{j_2+2},\dots,t_{n-j_1-1}$. We have 
	$$ \sum_{t_{j_2+2},\dots,t_{n-j_1-1}=0}^{1}T=\sum_{t_{j_2+2},\dots,t_{n-j_1-2}=0}^{1}1=2^{n-j_1-j_2-3}  $$
	and we know the sums $\sum_{t_{j_2+2},\dots,t_{n-j_1-1}=0}^{1}v(0)$ and $\sum_{t_{j_2+2},\dots,t_{n-j_1-1}=0}^{1}v(1)$ from above. Similarly as above we can show
	\begin{align*}
	   \sum_{t_{j_2+2},\dots,t_{n-j_1-1}=0}^{1}&(T\oplus R)v(0) \\
		   =& \frac12 \sum_{t_{j_2+3},\dots,t_{n-j_1-1}=0}^{1} (1\oplus (1-a_{j_2+2})(t_{j_2+3}\oplus \dots\oplus t_{n-j_1-1}\oplus R)) \\&+\sum_{l=0}^{2^{n-j_1-j_2-3}-1} \frac{l}{2^{n-j_1-j_2-2}}
	\end{align*}
	as well as
	\begin{align*}
	   \sum_{t_{j_2+2},\dots,t_{n-j_1-1}=0}^{1}&(T\oplus R)v(1) \\
		   =& \frac12 \sum_{t_{j_2+3},\dots,t_{n-j_1-1}=0}^{1} (1-a_{j_2+2})(t_{j_2+3}\oplus \dots\oplus t_{n-j_1-1}\oplus R\oplus 1)\\&+\sum_{l=0}^{2^{n-j_1-j_2-3}-1} \frac{l}{2^{n-j_1-j_2-2}},
	\end{align*}
	which yields
	$$ \sum_{t_{j_2+2},\dots,t_{n-j_1-1}=0}^{1}(T\oplus R)(v(1)+v(0))=2^{n-j_1-j_2-4}+2\sum_{l=0}^{2^{n-j_1-j_2-3}-1} \frac{l}{2^{n-j_1-j_2-2}}. $$
	Now we can combine our results with~\eqref{art4} to obtain the claimed result.
	\end{proof}

		\paragraph{Case 7: $\bsj\in\mathcal{J}_7:=\{(j_1,j_2): 0\leq j_1,j_2 \leq n-1 \textit{\, and \,} j_1+j_2\geq n-2\}$}
		
	\begin{proposition} Let $\bsj\in \mathcal{J}_7$ and $\bsm\in \DD_{\bsj}$.
	   Then we have $|\mu_{\bsj,\bsm}|\lesssim 2^{-n-j_1-j_2}$ for all $\bsm\in\DD_{\bsj}$ and $|\mu_{\bsj,\bsm}|=2^{-2j_1-2j_2-4}$ for all
		but at most $2^n$ elements $\bsm\in\DD_{\bsj}$.
	\end{proposition}
  
	\begin{proof}
	  At most $2^n$ of the $2^{|\bsj|}$ dyadic boxes $I_{\bsj,\bsm}$ for $\bsm\in\DD_{\bsj}$ contain points. For the empty boxes, only the linear part of the discrepancy function contributes
		to the corresponding Haar coefficients; hence $|\mu_{\bsj,\bsm}|=2^{-2j_1-2j_2-4}$ for all
		but at most $2^n$ elements $\bsm\in\DD_{\bsj}$. The non-empty boxes contain at most 4 points. Hence we find by~\eqref{art4}
		\begin{align*}
		  |\mu_{\bsj,\bsm}|\leq& 2^{-n-j_1-j_2-2}\sum_{\bsz \in \cP\cap I_{\bsj,\bsm}}|(1-|2m_1+1-2^{j_1+1}z_1|)(1-|2m_2+1-2^{j_2+1}z_2|)| \\&+2^{-2j_1-2j_2-4} \\
			 \leq& 2^{-n-j_1-j_2-2}4 +2^{-2j_1-2j_2-4}\leq 2^{-n-j_1-j_2}+2^{-j_1-j_2-(n-2)-4}\lesssim 2^{-n-j_1-j_2}.
		\end{align*}
	\end{proof}

		\paragraph{Case 8: $\bsj\in\mathcal{J}_8:=\{(n-1,-1),(-1,n-1)\}$}
		
	\begin{proposition} Let $\bsj\in \mathcal{J}_8$ and $\bsm\in \DD_{\bsj}$.
	  Let $\bsj=(n-1,-1)$ or $\bsj=(-1,n-1)$. Then $\mu_{\bsj,\bsm}\lesssim 2^{-2n}$.
	\end{proposition}

	\begin{proof}
	   At most 2 points lie in $I_{\bsj,\bsm}$. Hence, if $\bsj=(n-1,-1)$, then by~\eqref{art2} we have
		\begin{align*} |\mu_{\bsj,\bsm}|\leq& 2^{-n-j_1-1}\sum_{\bsz \in \cP\cap I_{\bsj,\bsm}}|(1-|2m_1+1-2^{j_1+1}z_1|)(1-z_2)|+2^{-2j_1-3} \\
		                     =& 2^{-n-j_1-1}2+2^{-2j_1-3}=2^{-2n+1}+2^{-2n-1}\lesssim 2^{-2n}.
		   \end{align*}
			The case $\bsj=(-1,n-1)$ can be shown the same way.
	\end{proof}
	
		\paragraph{Case 9: $\bsj\in\mathcal{J}_{9}:=\{(j_1,j_2): j_1\geq n \textit{\, or \,} j_2\geq n\}$}
	
	\begin{proposition} Let $\bsj\in \mathcal{J}_{9}$ and $\bsm\in \DD_{\bsj}$.
	 Then $\mu_{\bsj,\bsm}=-2^{-2j_1-2j_2-4}$.
	\end{proposition}
	
	\begin{proof}
	   The reason is that no point is contained in the interior of $I_{\bsj,\bsm}$ in this case and hence only the
		linear part of the discrepancy function contributes to the Haar coefficient in~\eqref{art4}.
	\end{proof}

\noindent{\bf Authors' Address:}

\noindent Ralph Kritzinger and Friedrich Pillichshammer, Institut f\"{u}r Finanzmathematik und angewandte Zahlentheorie, Johannes Kepler Universit\"{a}t Linz, Altenbergerstra{\ss}e 69, A-4040 Linz, Austria.\\
{\bf Email:} ralph.kritzinger(at)jku.at and friedrich.pillichshammer(at)jku.at


\begin{thebibliography}{10} 

\bibitem{bil} D. Bilyk, \textit{The $L^2$-discrepancy of irrational lattices}, in: Monte Carlo and Quasi-Monte Carlo Methods 2012, J. Dick, F.Y. Kuo, G.W. Peters and I.H. Sloan (eds.), Springer, Berlin Heidelberg New York, 2013, 289--296.


\bibitem{daven} H. Davenport, \textit{Note on irregularities of distribution}, Mathematika 3 (1956), 131--135.


\bibitem{DP10} J. Dick and F. Pillichshammer, \textit{Digital nets and sequences. Discrepancy theory and quasi-Monte Carlo integration}, Cambridge University Press, Cambridge, 2010.







\bibitem{FauPil} H. Faure and F. Pillichshammer, \textit{$L_p$ discrepancy of generalized two-dimensional Hammersley point sets}, Monatsh. Math. 158 (2009), 31--61.



\bibitem{hala} G. Hal\'{a}sz, \textit{On Roth's method in the theory of irregularities of point distributions}, in: Recent progress in analytic number theory, Vol. 2, Academic Press, London-New York, 1981, 79--94. 

\bibitem{HaZa} J.H. Halton and S.K. Zaremba, \textit{The extreme and $L^2$ discrepancies of some plane sets}, Monatsh. Math. 73 (1969), 316--328.


\bibitem{hin2010} A. Hinrichs, \textit{Discrepancy of Hammersley points in Besov spaces of dominating mixed smoothness}, Math. Nachr. 283 (2010), 478--488.

\bibitem{HKP14} A. Hinrichs, R. Kritzinger and F. Pillichshammer, \textit{Optimal order of $L_p$ discrepancy of digit shifted Hammersley point sets in dimension 2},  Unif. Distrib. Theory 10 (2015), 115--133.



\bibitem{Kritz}  R. Kritzinger, \textit{Finding exact formulas for the $L_2$ discrepancy of digital $(0,n,2)$-nets via Haar functions}, to appear in Acta Arith. (2018).



\bibitem{kuinie} L. Kuipers and H. Niederreiter, \textit{Uniform distribution of sequences}, John Wiley, New York, 1974.

\bibitem{lp01} G. Larcher and F. Pillichshammer, \textit{Walsh series analysis of the $L_2$ discrepancy of symmetrisized point sets}, Monatsh. Math. 132 (2001), 1--18.

\bibitem{Lar}  G. Larcher and F. Pillichshammer, \textit{Sums of distances to the nearest integer and the discrepancy of digital nets}, Acta Arith. 106 (2003), 379--408.

\bibitem{LP14} G. Leobacher and F. Pillichshammer, \textit{Introduction to Quasi-Monte Carlo Integration and Applications}, Compact Textbooks in Mathematics, Birkh\"auser, Cham, 2014.

\bibitem{lerch} M. Lerch, \textit{Question 1547.} L'Interm\'{e}diaire des Math\'{e}maticiens 11 (1904), 144--145.


\bibitem{Nied87} H. Niederreiter, \textit{Point sets and sequences with small discrepancy,} Monatsh. Math.
104 (1987), 273--337.

\bibitem{Nied92} H. Niederreiter, \textit{Random number generation and quasi-Monte Carlo methods}, Number 63 in CBMS-NFS Series in Applied Mathematics, SIAM, Philadelphia, 1992.

\bibitem{Pill}  F. Pillichshammer, \textit{On the $L_p$ discrepancy of the Hammersley point set}, Monath. Math. 136 (2002) 67--79.

\bibitem{Roth2} K.F. Roth, \textit{On irregularities of distribution}, Mathematika 1 (1954), 73--79.

\bibitem{Schm72distrib} W.M. Schmidt, \textit{Irregularities of distribution VII}, Acta Arith. 21 (1972), 45--50.

\bibitem{schX} W.M. Schmidt, \textit{Irregularities of distribution. X}, in: Number Theory and Algebra, Academic Press, New York, 1977, 311--329.

\bibitem{Vi} I.V. Vilenkin, \textit{Plane nets of integration}, \v{Z}. Vy\v{c}isl. Mat. i Mat. Fiz. 7 (1967), 189--196. English translation
in: U.S.S.R. Comput. Math. Math. Phys. 7(1) (1967), 258--267.

\bibitem{Warn} T.T. Warnock, \textit{Computational investigations of low discrepancy point sets}, in: \textit{Applications of
Number Theory to Numerical Analysis,} Academic Press, 1972.

\end{thebibliography}
\end{document}